\documentclass[11pt, a4paper]{amsart}
\usepackage{amsfonts,amsmath,amssymb, amscd}
\usepackage{txfonts}
\usepackage[all]{xy}

\newtheorem{theorem}{Theorem}[section]
\newtheorem{lemma}[theorem]{Lemma}
\newtheorem{prop}[theorem]{Proposition}
\newtheorem{corollary}[theorem]{Corollary}

\theoremstyle{definition}

\newtheorem{rem}[theorem]{Remark}

\newtheorem{notation}[theorem]{Notation}
\newtheorem{example}[theorem]{Example}

\usepackage{color}

\newcommand{\lcosmash}{\mathop{\raisebox{0.2ex}{\makebox[0.92em][l]{${\scriptstyle>\mathrel{\mkern-4mu}\blacktriangleleft}$}}}}

\newcommand{\rboson}{\mathop{\raisebox{0.12ex}{$\shortmid$}\hspace{-0.8mm}\raisebox{0.2ex}{\makebox[0.86em][r]{${\scriptstyle\gtrdot\mathrel{\mkern-4mu}<}$}}}}
\newcommand{\lboson}{\mathop{\raisebox{0.2ex}{\makebox[0.86em][l]{${\scriptstyle>\mathrel{\mkern-4mu}\lessdot}$}}\raisebox{0.12ex}{\hspace{-0.8mm}$\shortmid$}}}

\usepackage{mathrsfs}

\renewcommand{\k}{\Bbbk}
\newcommand{\ot}{\otimes}
\newcommand{\G}{\hat{G}}
\renewcommand{\L}{L^*}
\newcommand{\R}{R^*}
\newcommand{\sL}{\mathscr{L}}
\newcommand{\sI}{\mathscr{I}}
\newcommand{\sM}{\mathscr{M}}
\newcommand{\YD}{\mathscr{Y\hspace{-1mm}D}}
\renewcommand{\a}{\alpha}
\renewcommand{\b}{\beta}
\newcommand{\tl}{\triangleleft}
\newcommand{\tr}{\triangleright}

\newcommand{\Sm}{S^{-1}}
\renewcommand{\l}{\langle}
\renewcommand{\r}{\rangle}
\newcommand{\ep}{\varepsilon}
\newcommand{\hi}{h_{(1)}}
\newcommand{\hn}{h_{(2)}}
\newcommand{\hs}{h_{(3)}}
\newcommand{\ai}{\alpha_{(1)}}
\newcommand{\an}{\alpha_{(2)}}
\newcommand{\as}{\alpha_{(3)}}

\newcommand{\isomto}{\overset{\simeq}{\longrightarrow}}

\numberwithin{equation}{section}

\begin{document} 

\title[Finite quantum groups]{Simple modules over finite quantum groups and\\ their Drinfel'd doubles}

\author[A.~Masuoka]{Akira Masuoka}
\address{Akira Masuoka: 
Institute of Mathematics, 
University of Tsukuba, 
Ibaraki 305-8571, Japan}
\email{akira@math.tsukuba.ac.jp}

\author[A.~Nakazawa]{Atsuya Nakazawa}
\address{Atsuya Nakazawa: 
Nissin Computer System Co. Ltd., 
1-5-3 Koraku, Bunkyo-ku, Tokyo 112-0004, Japan}
\email{nakazawa@nisin.co.jp}

% \pagewiselinenumbers

\begin{abstract}
By finite quantum groups we mean Lusztig's finite-dimensional pointed Hopf algebras
called quantum Frobenius Kernels \cite{L,Lu}, and their natural generalizations 
due to Andruskiewitsch and Schneider \cite{AS1,AS2}. For a Hopf algebra $H$ 
in a special class of the latter generalizations, which arises  
from a pair of quantum linear spaces,
Krop and Radford \cite{KR} described the simple 
modules over $H$ and over the Drinfel'd double $D(H)$, showing that 
they fall into a simple pattern of parametrization. 
We extend the description to a wider class of Hopf algebras
which includes the quantum Frobenius Kernels, renewing
the parametrization pattern so as to connect directly to
the so-called triangular decomposition. 
\end{abstract}

\maketitle

\noindent
{\sc Key Words:}
finite quantum group, Drinfel'd double, Hopf algebra, Nichols algebra, 
pre-Nichols algebra.

\medskip
\noindent
{\sc Mathematics Subject Classification (2010): 16T05, 17B37.}

\section{Introduction}\label{sec:intro}

Lusztig \cite{L,Lu} constructed the finite-dimensional pointed Hopf algebras $H$
called quantum Frobenius Kernels, which admit as their feature the triangular decomposition
\begin{equation}\label{eq:triangular_decomp1}
H= H_- \ot H_0 \ot H_+,
\end{equation}
just as the quantum universal enveloping algebras due to Drinfel'd and Jimbo do. 
Here the $0$-part $H_0$ is the group algebra $\k G$ of some finite abelian group $G$. 
As was shown by Lusztig, the simple modules over a Hopf algebra $H$ as above is 
parametrized by the dual group $\G$ of $G$. This parametrization falls into a quite
simple pattern that will be formulated by Theorem \ref{thm:pattern} below. 
We emphasize that the triangular decomposition plays a role for it. 

Andruskiewitsch and Schneider \cite{AS2} classified the finite-dimensional pointed 
Hopf algebras with abelian group of grouplikes, 
under mild restriction to dimensions. Essential for their result 
are the Hopf-algebra objects, called \emph{Nichols algebras}, in the braided tensor category 
${}^G_G\YD$ of Yetter-Drinfel'd modules over a finite abelian group $G$. In fact, their principle 
is first to classify the finite-dimensional Nichols algebras $L$, and then to classify those
Hopf algebras which turn, after $\mathrm{gr}$ applied, into the bosonizations $L \lboson \k G$ of $L$ by $\k G$. 
The finite-dimensional pointed Hopf algebras of our concern are those $H$ which are
deformed from $(L \ot R) \lboson \k G$ by two-cocycle, as in \cite{M} (see also \cite{Ma}), 
so that $H$ admits (obviously by construction) the triangular decomposition 
\begin{equation}\label{eq:triangular_decomp2}
H= L \ot \k G \ot R
\end{equation}
which generalizes \eqref{eq:triangular_decomp1}. 
Here, $L$ and $R$ are finite-dimensional
Nichols algebras (in positive characteristic, they may be \emph{pre-Nichols algebras} \cite{M}, more generally)
which are mutually symmetric in that sense that the braidings between $L \ot R$ and $R \ot L$ are mutually
inverse. The explicit construction of $H$ is given by Proposition \ref{prop:def_of_H}. 
Krop and Radford \cite{KR} described, among others,  
the simple modules over $H$ and over the Drinfel'd double $D(H)$,
showing that they fall into a parametrization pattern (see the next paragraph),  
when $L$ and $R$ are quantum linear spaces; the result for $D(H)$, in particular, was 
surprising to the authors. 
The aim of this paper is to extend these results by Krop and
Radford to our $H$; see Theorems \ref{thm:Hsimple} and \ref{thm:D(H)simple}. 

We actually modify the construction in \cite{M}, taking the
$R$ above to be a (pre-) Nichols algebra in the category $\YD^G_{G}$ of 
the opposite-sided Yetter-Drinfel'd modules; it gives rise to the opposite-sided 
bosonization $\k G \rboson R$. This modification
indeed realizes the triangular decomposition \eqref{eq:triangular_decomp2},
and makes a more conceptual treatment of $D(H)$ possible. Our parametrization pattern 
for simple modules (see Theorem \ref{thm:pattern}) renews that by Krop and Radford so as
to connect directly to the triangular decomposition.

The pointed Hopf algebras investigated by Krop and Radford, as well as Lusztig's
quantum Frobenius Kernels, are reproduced from our $H$ in Example \ref{ex:two_examples}.  
\medskip

In revision of the manuscript the authors added the articles \cite{BT}, \cite{HN} and \cite{PV}
to the References, 
among which Bellamy and Thiel \cite{BT} develops highest weight theory of $\mathbb{Z}$-graded
modules over a finite-dimensional $\mathbb{Z}$-graded algebra with triangular decomposition. What they call
\emph{triangular decomposition} (attributed to Holmes and Nakano \cite{HN}) 
essentially generalizes what we will define in Section 
\ref{subsec:pattern}, and turns
out to simplify some of our arguments. 
See Remark \ref{rem:added_in_revision} for more details which contain our contribution realized in the
end.  
\medskip

Throughout in what follows we work in the situation given below, unless otherwise stated. 

\begin{notation}\label{notation1}
Let $G$ be a finite abelian group, whose identity element will be denoted by $1$.
Let $\k$ be the field over which we work; the unadorned $\ot$ denotes the tensor product over $\k$.
We only assume that the characteristic 
$\mathrm{chrar}~\k$ of $\k$
does not divide the order $\mathrm{ord}(G)$ of $G$, and $\k$ contains a primitive $N$-th root of $1$, where
$N=\mathrm{exp}(G)$ denotes the exponent of $G$. It follows that the dual group $\G$ of $G$, which consists 
of all group homomorphisms from $G$ to the multiplicative group $\k^{\times}=\k \setminus \{0\}$ of $\k$,
is non-canonically isomorphic to $G$, and the dual Hopf algebra $(\k G)^*$ of the group algebra
$\k G$ is canonically isomorphic to $\k \G$. Note that the Hopf algebras are both semisimple. 
\end{notation}

%%%%%%%%%%%%%%%%%%%%%%%%%%%%%%%% Section 2 %%%%%%%%%%%%%%%%%%%%%%%%%%%%%%%%%%%%%%%%%%%%%%%%

\section{A parametrization pattern for simple modules}\label{sec:pattern}

\subsection{}\label{subsec:tensor_product}
Let $\Lambda$ be a finite-dimensional algebra. Two subalgebras $L$ and $R$ of $\Lambda$ are said to be
\emph{permutable}, if $LR=RL$. In this case $LR\, (=RL)$ is a subalgebra of $\Lambda$. More than two 
subalgebras are said to be \emph{permutable}, if any two of them are permutable.

Let $R_1,R_2,\dots, R_n$ be subalgebras of $\Lambda$. We write
\begin{equation}\label{eq:presentation}
\Lambda = R_1 \ot R_2 \ot \dots \ot R_n, 
\end{equation}
if the product map\ 
$R_1\ot R_2 \ot \dots \ot R_n \to \Lambda,\ x_1\ot x_2 \ot \dots \ot x_n \mapsto x_1x_2\dots x_n$\
is a bijection. Suppose that this is the case.
We call \eqref{eq:presentation} a \emph{tensorial presentation} of $\Lambda$.
Given numbers $1 \le i_1 < i_2 < \dots < i_r \le n$, 
we regard $R_{i_1}\ot R_{i_2} \ot \dots \ot R_{i_r}$
as a subspace of $\Lambda$, identifying it with $R_{i_1}R_{i_2}\dots R_{i_r}$ through the 
isomorphic product map $R_{i_1}\ot R_{i_2} \ot \dots \ot R_{i_r}\isomto R_{i_1}R_{i_2}\dots R_{i_r}$. 
If $R_{i_1}, R_{i_2}, \dots, R_{i_r}$ are permutable, then
$\Gamma:= R_{i_1}\ot R_{i_2} \ot \dots \ot R_{i_r}$ is a subalgebra of $\Lambda$, and
\[
\Gamma=R_{i_{\sigma(1)}}\ot R_{i_{\sigma(2)}} \ot \dots \ot R_{i_{\sigma(r)}}
\]
gives another tensorial presentation of $\Gamma$ for every permutation $\sigma \in \mathfrak{S}_r$, as
is seen by counting dimensions. 

\subsection{}\label{subsec:pattern}
Let $\Lambda$ be a finite-dimensional algebra which includes, as subalgebras, 
two algebras $L$ and $R$, as well as the group algebra $\k G$, so that
\begin{equation}\label{eq:TRD}
\Lambda = L \ot \k G \ot R. 
\end{equation}
Assume that $L$ and $\k G$, as well as $\k G$ and $R$, are permutable, so that
\[
A:= L \ot \k G,\quad B:= \k G \ot R 
\]
are subalgebras of $\Lambda$. We emphasize that 
each of these is not required to be, as an algebra, the tensor product of the two tensor factors. 
We assume in addition, 
\begin{itemize}
\item[(I)] $L$ and $R$ are augmented, and their augmentation ideals $L^+$ and $R^+$ are
nilpotent (or in other words, unique maximal ideals),
and 
\item[(II)] $(\k G) L^+ = L^+ (\k G)$ in $A$, and $R^+(\k G) = (\k G) R^+$ in $B$. 
\end{itemize}
As for (II) we remark that the first equality, for example, holds if $(\k G) L^+ \subset L^+ (\k G)$
or $(\k G) L^+ \supset L^+ (\k G)$, as is seen by counting dimensions. 
The assumptions imply that
\[
I_A:= L^+\ot \k G,\quad I_B:=\k G \ot R^+
\]
are ideals of $A$ and of $B$, and are, moreover, the Jacobson radicals so that
\[
A/I_A = \k G = B/I_B. 
\]
Every simple left or right $A$-module (resp.,~$B$-module)
then arises uniquely from a group homomorphism $\lambda : G \to \k^{\times}$ (or an element
of the dual group $\G$); it is defined on $\k$ by the algebra maps 
\[ 
A \overset{\mathrm{proj}}{\longrightarrow} A/I_A = \k G \overset{\k \lambda}{\longrightarrow} \k \quad 
(\text{resp.,}\  B \overset{\mathrm{proj}}{\longrightarrow} B/I_B = \k G \overset{\k \lambda}{\longrightarrow} k), 
\]
where $\mathrm{proj}$ denotes the natural projection, and $\k \lambda$ denotes the 
$\k$-linearization of $\lambda$. We denote the simple module by
\[
\k_A(\lambda)\quad (\text{resp}.,~~\k_B(\lambda)). 
\]

\begin{lemma}\label{lem:lambdamu}
Given $\lambda, \mu \in \G $, we have
\[
\k_A(\lambda) \otimes_A \Lambda \otimes_B \k_B(\mu)=
\begin{cases}
\, \k &\text{if}\hspace{2mm} \lambda=\mu,\\
\, 0  &\text{if}\hspace{2mm} \lambda\ne\mu.
\end{cases}
\]
\end{lemma}
\begin{proof}
This follows since we have
\[
\k_A(\lambda) \otimes_A \Lambda \otimes_B \k_B(\mu)=
\k_A(\lambda) \otimes_{\k G} \k G \otimes_{\k G} \k_B(\mu)= 
\k_A(\lambda) \otimes_{\k G} \k_B(\mu).
\]
\end{proof}

Let $\lambda \in \G$. We define a left $\Lambda$-module by
\[
\sM(\lambda) = \Lambda \ot_B \k_B(\lambda).
\]
This is isomorphic to $L$ as a left $L$-module, and is non-zero, in particular. 

\begin{prop}\label{prop:Llambda}
$\sM(\lambda)$ includes the largest proper $\Lambda$-submodule,
which we denote by $\sI(\lambda)$.  
\end{prop}
\begin{proof}
This follows, since
every proper $\Lambda$-submodule of $\sM(\lambda)$ is included in the largest 
proper $L$-submodule \ $L^+\hspace{-0.6mm}\sM(\lambda)\, (\simeq L^+)$. 
\end{proof}

We define a simple left $\Lambda$-module by
\[
\sL(\lambda) = \sM(\lambda)/\sI(\lambda). 
\]
Let 
\[
\Lambda\text{-}\mathsf{Simple}
\]
denote the set of all isomorphism classes of simple left $\Lambda$-modules. 

\begin{theorem}\label{thm:pattern}
$\lambda \mapsto \sL(\lambda)$ gives a bijection $\G \isomto \Lambda\text{-}\mathsf{Simple}$.   
\end{theorem}
\begin{proof}
To prove the surjectivity of the map, let $M$ be a simple left $\Lambda$-module.
Choose a simple $B$-submodule of $M$ and its non-zero element $m$. 
Then there uniquely exists $\lambda\in \G$
such that $\k_B(\lambda) \simeq B m$ as left $B$-modules. The $\Lambda$-module map 
$\Lambda \to M$ sending $1$ to $m$ induces a $\Lambda$-module map $\sM(\lambda) \to M$, which is 
surjective by simplicity of $M$. We have thus $M \simeq \sL(\lambda)$.

If $\mu \ne \lambda$ in $\G$, it follows by Lemma \ref{lem:lambdamu} that $\k_A(\mu) \otimes_A \sL(\lambda)=0$. 
Nakayama's Lemma applied to the left $A$-module $\sL(\lambda)$ shows that 
$\k_A(\lambda) \otimes_A \sL(\lambda)\ne 0$. Therefore,
$\sL(\lambda)$ is the unique (up to isomorphism) simple left $\Lambda$-module $M$ such that $\k_A(\lambda)\ot_A M \ne 0$. 
This shows the injectivity of the map. 
\end{proof}

%%%%%%%%%%%%%%%%%%%%%%%%%%%%%% Section 3 %%%%%%%%%%%%%%%%%%%%%%%%%%%%%%%%%%%%%%

\section{Simple modules over finite quantum groups}\label{sec:quantum_group}

\subsection{}\label{subsec:YD}
In this subsection the base field $\k$ may be arbitrary, and $G$ is an abelian group
which may not be finite. 
 
A $G$-graded vector space $V$ is presented as $V = \bigoplus_{g\in G}V_g$ with
$g$-component $V_g$. A left $\k G$-module structure on a vector space $V$ is 
presented as $g \tr v$, where $g \in G$ and $v \in V$. A \emph{left Yetter-Drinfel'd module}
over $G$ is a $G$-graded vector space $V = \bigoplus_{g\in G}V_g$ which is endowed with a left 
$\k G$-module structure such that
\[
f \tr V_g =V_g \ \text{for all} \ f, g \in G.
\]
See \cite[Definition 1.2, p.6]{AS1}. 
The left Yetter-Drinfel'd modules over $G$ form a category, ${}^G_G\YD$, whose morphisms
are $\k G$-linear $G$-graded maps. 

Let $V, \ W \in {}^G_G\YD$. Then the
tensor product $V \ot W$ over $\k$ is naturally an object of the category so that
\[
(V \ot W)_g = \bigoplus_{f\in G}V_f \ot W_{f^{-1}g}, \quad
g\tr(v\ot w) = (g\tr v)\ot(g\tr w),
\]
where $g \in G$, \ $v \in V$ \ and \ $w \in W$. With respect to this tensor product
${}^G_G\YD$ is a tensor category; the unit object is $\k$ that is trivially $G$-graded,
$\k =(\k)_1$, and is given the trivial $\k G$-module structure $g \tr 1=1$,\ $g \in G$. 
Moreover, the tensor category is braided with respect to the braiding
\begin{equation}\label{eq:braiding}
c=c_{V,W} : V \ot W \isomto W \ot V,\quad c(v \ot w) = (g \tr w) \ot v,
\end{equation}
where $v \in V_g$ with $g \in G$, and $w \in W$. 

Since ${}^G_G\YD$ is thus a braided tensor category, the concept of Hopf-algebra objects in it is naturally
defined; see \cite[Sect.~10.5]{Mon}. 
Let $L$ be such an object. Its antipode $s : L \to L$ is the convolution-inverse
of the identity map $\mathrm{id}$ on $L$, which is necessarily a morphism of the category. 
Since $L$ is a left $\k G$-module algebra, the tensor product $L \ot \k G$ turns into the ($G$-graded)
algebra $L\rtimes \k G$ of smash (or semi-direct) product. 
Since $L$ is a $G$-graded coalgebra,
$L \ot \k G$ turns into the (right $\k G$-module) coalgebra $L \lcosmash \k G$ of smash coproduct;
see \cite[Definition 10.6.1, p.207]{Mon}.  Its structure is given by
\[
\Delta(x \ot f) = \sum_i (x_i \ot  g_i f) \ot (y_i \ot f), \quad \ep(x \ot f) =\ep(x),
\]
where $x \in L$,\ $f \in G$, and $\Delta(x)=\sum_i x_i \ot y_i$ is the coproduct on $L$ 
with $y_i \in L_{g_i}$, $g_i \in G$. 
Here and in what follows, $\ep$ denotes the counit of any coalgebra.
With the algebra and coalgebra structures above, $L \ot \k G$ turns into
an ordinary Hopf algebra, whose antipode $S$ is given by
\[
S(x \ot f) = (1 \ot f^{-1}g^{-1})(s(x)\ot 1)\ \, (= s((gf)^{-1}\tr x)\ot (gf)^{-1}), 
\]
where $x \in L_g$ with $g \in G$, and $f \in G$. This Hopf algebra is denoted by
\[
L \lboson \k G, 
\]
and is called the \emph{bosonization} of $L$; this was first constructed by Radford \cite{R} 
under the name ``biproduct", and then renamed as above by Majid \cite{Maj}.  
This Hopf algebra has the Hopf algebra maps from and to $\k G$
\begin{equation}\label{eq:triple}
\iota : \k G \to L \lboson \k G, \ \iota(g)=1\ot g,\quad \pi = \ep \ot \mathrm{id} : 
L \lboson \k G \to \k G 
\end{equation}
such that $\pi \circ \iota =\mathrm{id}$. Every triple $(H,\iota, \pi)$ that consists of 
a Hopf algebra $H$ and Hopf algebra maps $\iota$, $\pi$ as above arises uniquely (up to isomorphism)
from a Hopf-algebra object in ${}^G_G\YD$, in this way; see 
\cite[Theorem 3]{R} and \cite[Proposition 1.1]{CDMM}.

Let $V \in {}^G_G\YD$. The tensor algebra $T(V)$ on $V$ uniquely turns into
a Hopf-algebra object in ${}^G_G\YD$ so that every element $v$ of $V$ is primitive, or in notation,
$\Delta(v)=v\ot 1 + 1 \ot v$. Since this has the trivial coradical $\k$, every 
quotient bialgebra object is necessarily a Hopf-algebra object. 
We say that
$T(V)=\bigoplus_{n\ge 0}T^n(V)$ is \emph{graded} by $\mathbb{N}=\{ 0,1,2,\dots \}$, meaning that 
every homogeneous component is a sub-object in ${}^G_G\YD$, and it is $\mathbb{N}$-graded
as an algebra and as a coalgebra. A quotient Hopf-algebra object $T(V)/I$ of $T(V)$ such that
$I$ is $\mathbb{N}$-homogeneous with $I \cap V = \{ 0 \}$ is called a \emph{pre-Nichols algebra}
 \cite[Definition 2.1]{M} of $V$; it is characterized as an $\mathbb{N}$-graded Hopf-algebra
object $L = \bigoplus_{n\ge 0}L(n)$ with $L(0)=\k$, $L(1)=V$, that is generated by $V$. 
It is called the \emph{Nichols algebra} \cite[p.16]{AS1}
of $V$, if $I$ is the largest possible, or equivalently if
any homogeneous component of degree $>1$ does not contains a non-zero primitive. 

A \emph{right Yetter-Drinfel'd module}
over $G$ is a $G$-graded vector space $V = \bigoplus_{g\in G}V_g$ which is endowed with a right 
$\k G$-module structure, $v \tl g$, such that 
$V_g  \tl f=V_g$  for all  $f, g \in G$.
The right Yetter-Drinfel'd modules over $G$ form a category, $\YD^G_G$, which is again 
a braided tensor category. Since $G$ is abelian, it is identified with ${}^G_G\YD$
as a tensor category, but has the distinct braiding
\[
c'=c'_{V,W} : V \ot W \isomto W \ot V,\quad c'(v \ot w) = w \ot (v\tl g),
\]
where $v \in V$, and $w \in W_g$ with $g \in G$. 
The Hopf algebra which naturally arises from a Hopf-algebra object $R$ in $\YD^G_G$ is denoted by
\[ \k G \rboson R, \]
and is called again the \emph{bosonization} of $R$.
(Pre-)Nichols algebras in $\YD^G_G$ are defined as before.

It is known that there exists a braided tensor-equivalence between
the left and the right Yetter-Drinfel'd modules in a generalized situation; see
\cite[Proposition 2.2.1]{AG}. This fact, however, will not be used in the sequel. 

The following important result is due to Iv\'{a}n Angiono \cite{An}. 

\begin{theorem}[\text{\cite[Theorem 4.13]{An}}]\label{thm:in_char_zero}
Assume $\mathrm{char}\, \k = 0$, and that the abelian group $G$ is finite. Then every finite-dimensional 
pre-Nichols algebra in ${}_G^G\YD$ or in $\YD{}_G^G$ is necessarily Nichols. 
\end{theorem}

\subsection{}\label{subsec:H}
Return to the situation given by Notation \ref{notation1}. 
In addition assume, throughout in what follows, that we are in the situation given below. 

\begin{notation}\label{notation2}
We let
\[ 
V \in {}^G_G\YD , \quad W \in \YD^G_G
\]
be finite-dimensional objects in the respective Yetter-Drinfel'd module
categories such that
\begin{equation}\label{eq:symmetric}
(w\tl f) \otimes (g \tr v)=w \ot v \ \, \text{on} \ \, W_g \ot V_f
\end{equation}
for all $f, g \in G$, $v \in V_f$ and $w \in W_g$. Choose arbitrarily a linear map
\begin{equation}\label{eq:delta}
\delta : W \ot_{\k G} V \to \k.
\end{equation}
Let us write $\delta(w, v)$ for $\delta(w\ot v)$; it then holds that\ 
$\delta(w \tl g, v)=\delta(w, g \tr v)$,\ $g \in G$. 
Suppose that we are given 
finite-dimensional pre-Nichols algebras,
\[ 
L \ \, \text{of} \ \, V \ \, \text{in} \ \, {}^G_G\YD ,
\quad R \ \, \text{of} \ \, W \ \, \text{in} \ \,  \YD^G_G.
\]
By Angiono's Theorem \ref{thm:in_char_zero} these are necessarily Nichols, if $\mathrm{char}\, \k =0$. 
\end{notation}

\begin{prop}\label{prop:def_of_H}
There uniquely exists a Hopf algebra which satisfies: 
\begin{itemize}
\item[(i)]~~$H$ includes $L$, $\k G$ and $R$ as subalgebras so that
\begin{equation}\label{eq:tensor_pres_H}
H = L \ot \k G \ot R; 
\end{equation}
\item[(ii)]~~In $H$, \ $L \ot \k G$ is a Hopf subalgebra which coincides with
the bosonization $L \lboson \k G$ of $L$, and $\k G \ot R$ is a Hopf subalgebra
which coincides with the bosonization $\k G \rboson R$ of $R$; 
\item[(iii)]
We have
\begin{equation}\label{eq:wv}
w v = v w + \delta(w, g \tr v)\, (f-g) \ \, \text{in} \ \, H 
\end{equation}
for all $f, g \in G, \ v \in V_f$ and $w \in W_g$. 
\end{itemize}
\end{prop}

\begin{proof}
The uniqueness of the Hopf algebra is obvious. 
We prove the existence.  

Let $K := \k G \rboson R$. This Hopf algebra is $\mathbb{N}$-graded so that $K(n)=\k G \ot R(n)$, $n \ge 0$. 
In $K$, let
\[
U=S(W),\quad P=S(R)
\]
be the images of $W$ and of $R$, respectively, by the antipode $S$. Note that $S(w)=-wg^{-1}$ if $w \in W_g$. 
Moreover, $S$ is $\mathbb{N}$-graded, and $P$ coincides with the left coideal subalgebra
(see Remark \ref{rem:coideal_subalgebra} below)
\[
K^{\mathrm{co} \k G}=\{ x \in K \mid (\mathrm{id}\ot \pi)\circ \Delta(x)= x \ot \pi(1) \}, \]
of right $\k G$-coinvariants in $K$, 
where $\pi=\ep \ot \mathrm{id} : H= R\ot \k G \to \k G$ denotes the natural projection onto $\k G$. 
It then follows that
\[
U = \bigoplus_{g\in G} U_{g^{-1}}, \ \, \text{where}\ \, U_{g^{-1}}=\{ wg^{-1} \mid w \in W_g\},
\]
and this $U$ is an object in ${}_G^G\YD$, which has $U_g$ as the $g$-component, and whose left
$\k G$-module structure is defined by
\[
f \tr wg^{-1} := f(wg^{-1})f^{-1} \, (=(w \tl f^{-1})g^{-1}),\ \, f \in G, \ w \in W_g, \ g \in G.
\]
(One should note that the expression $w g^{-1}$ of every element of $U$ is unique.) 
Moreover, $P$ is a pre-Nichols algebra of $U$ such that $K= P \lboson \k G$; recall from 
the sentence following \eqref{eq:triple} the construction of triples.

We see that the objects $V$ and $U$ in ${}_G^G\YD$ are \emph{symmetric} in the sense 
of \cite[Definition 2.1]{M}, or namely, the braidings $c_{V,U}$ and $c_{U,V}$ are inverse to each other;
see \eqref{eq:braiding}. 
Indeed, if we choose $v \in V_f$ and $wg^{-1} \in U_{g^{-1}}$ with $w \in W_g$, then
\[ c_{U,V}\circ c_{V,U}(v \ot wg^{-1}) = c_{U,V}((f\tr wg^{-1})\ot v) = (g^{-1} \tr v)\ot (w\tl f^{-1})g^{-1}, \]
which coincides with $v \ot wg^{-1}$ by \eqref{eq:symmetric}. 

We wish to apply \cite[Theorem 3.10]{M} to the pre-Nichols algebras of the mutually symmetric
$V$ and $U$, choosing appropriately $\lambda : U \ot V \to \k$ as the
map given in \cite[Notation 3.1]{M}. We define $\lambda : U \ot V \to \k$
by
\[ 
\lambda(w g^{-1} \ot v) := \delta(w,~v),\quad v \in V, \ w \in W_g, \ g \in G. 
\]
This is a left $\k G$-module map with $\k$ regarded as the trivial $\k G$-module, since 
we have that for $f \in G$, 
\[
\lambda((f\tr w g^{-1})\ot (f\tr v))= \delta(w \tl f^{-1},~f \tr v) = \delta(w,~f^{-1}f \tr v)
=\lambda(w g^{-1}, v). 
\]
In view of \cite[Eq.~(3.4)]{M}, we compute in $T(V \oplus U)\lboson \k G$, 
\begin{align*}
&wg^{-1} \ot v - c_{U,V}(w g^{-1} \ot v) + \lambda(wg^{-1} \ot v) - g^{-1}f\, \lambda(wg^{-1} \otimes v)\\
&\quad = wg^{-1} \ot v - (g^{-1}\tr v)\ot wg^{-1} + \delta(w,~v)\, (1-g^{-1}f),
\end{align*}
where $v \in V_f$ and $w \in W_g$.
It now follows from \cite[Theorem 3.10]{M} that there uniquely exists a Hopf algebra $H$ which satisfies:
\begin{itemize}
\item[(i)]~~$H$ includes $L$ and $\k G \rboson R\, (=\k G \ot R)$ as subalgebras so that 
\[
H = L \ot (\k G \rboson R) \ (= L \ot \k G \ot R);
\]
\item[(ii)] In $H$, \ $\k G \rboson R$ is a Hopf subalgebra, and the subalgebra $L \ot \k G$ is a Hopf
subalgebra which coincides with $L \lboson \k G$; 
\item[(iii)] In $H$ \ we have
\[ wg^{-1}v - (g^{-1}\tr v)wg^{-1} + \delta(w,v)\, (1-g^{-1}f)=0 \]
or equivalently, 
\[ \big[w (g^{-1}\tr v) - (g^{-1}\tr v) w - \delta(w,\ v)\, (f-g)\big]\, g^{-1} =0 \]
for all $f, g \in G, \ v \in V_f$ and $w \in W_g$. 
\end{itemize}
We see that this $H$ is what is required. 
\end{proof}

\begin{rem}\label{rem:coideal_subalgebra}
Recall from \cite[Sect.~1]{T} that for a Hopf algebra $K$ in general, 
a \emph{left} (resp., \emph{right}) \emph{coideal subalgebra} is a subalgebra $J\subset K$ that
is a left (resp., right) coideal, $\Delta(J) \subset K \ot J$ (resp., $\Delta(J) \subset J \ot K$).
The dual notion is a \emph{quotient left} (resp., \emph{right}) $K$-\emph{module coalgebra} of $K$.
\end{rem} 

\begin{example}\label{ex:two_examples}
Let $n >0$ be an integer. Let $V$ and $W$ be vector spaces with bases 
$(v_i)_{1\le i \le n}$ and $(w_i)_{1\le i \le n}$, respectively. 
For each $1\le i \le n$, choose elements $f_i, g_i \in G$ and $\chi_i \in \G$. 
We let $V\in {}_G^G\YD$ and $W \in \YD_G^G$, defining 
\[
g\tr v_i=\chi_i(g)\, v_i, \ \, v_i \in V_{f_i};\quad 
w_i\tl g=\chi_i(g)\, w_i, \ \, w_i \in W_{g_i},
\]
where $g \in G$.
Note that every $v_i$ or $w_i$ spans a sub-object.
One sees that Eq.~\eqref{eq:symmetric} holds if and only if 
\begin{equation}\label{eq:chi}
\chi_i(f_j)\, \chi_j(g_i)=1 \ \, \text{for all} \ \, 1\le i,j \le n. 
\end{equation} 
For each $1\le i \le n$,
choose arbitrarily an element $c_i\in \k$, and let $\delta(w_i,v_j)=\delta_{i,j}c_i$, \ $1\le i, j\le n$. 
This indeed defines $\delta : W \ot_{\k G} V \to \k$. 

If \eqref{eq:chi} is satisfied, then the Hopf algebra $H$ as above is constructed from $L$, $R$ and $\delta$. 
Note that in this $H$, we have for every $g\in G$, and $1 \le i \le n$,
\begin{align*}
g v_i&=\chi_i(g)\, v_ig,\quad \Delta(v_i)=f_i\ot v_i + v_i \ot 1;\\
w_i g&=\chi_i(g)\, g w_i,\hspace{2.5mm} \Delta(w_i)=1 \ot w_i + w_i \ot g_i.
\end{align*}
\begin{itemize}
\item[(1)]
Assume \eqref{eq:chi}, or 
$\chi_i(f_j)=\chi_j(g_i)^{-1}$ for all $1\le i,j \le n$. 
Assume that the Yetter-Drinfel'd modules $\k x_i$, $1 \le i \le n$
(or equivalently,
$\k y_i$, $1 \le i \le n$) are pairwise symmetric, or explicitly, 
\[
\chi_i(f_j)\, \chi_j(f_i) = 1 \ \, \text{for all} \ \, 1\le i<j \le n. 
\]
We assume in addition, that $\chi_i(f_i) \ne 1$ for all $1\le i \le n$, and 
let $m_i=\mathrm{ord}(\chi_i(f_i))\, (> 1)$ denote its order in $\k^{\times}$; it equals
$\mathrm{ord}(\chi_i(g_i))$. Then the Nichols algebra $L$ of $V$ 
is generated by $v_1, \dots, v_n$, and is defined by
\[
v_jv_i = \chi_i(f_j)\, v_iv_j, \ 1 \le i<j \le n;\quad v_i^{m_i}=0, \ 1 \le i \le n. 
\]
This is indeed the braided tensor product 
\begin{equation}\label{eq:braided_tensor_product}
\k[v_1]/(v_1^{m_1})\ot \dots \ot \k[v_n]/(v_n^{m_n})
\end{equation}
of the Nichols algebras $\k[v_i]/(v_i^{m_i})$ of $\k v_i$, $1\le i \le n$. The Nichols algebra
$R$ of $W$ is described in the same way, with $v_i$ replaced by $w_i$. Our Hopf algebra
$H$ then coincides with the Hopf algebra $H(\mathscr{D})$ 
constructed by Krop and Radford \cite[Sect.~3.2, p.2579]{KR}, 
provided we choose $\delta$ so that $c_i\ne 0$ for all $1\le i \le n$. 

Suppose $\mathrm{char}\, \k = p >0$. If each $m_i$ above is supposed to be the
multiple $p^{e_i}\mathrm{ord}(\chi_i(f_i))$ of $\mathrm{ord}(\chi_i(f_i))$ by some power $p^{e_i}$, $e_i \ge0$
of $p$, then 
\eqref{eq:braided_tensor_product} gives a finite-dimensional pre-Nichols algebra of $V$, since
$v_i^{m_i}$ remains primitive in $\k[v_i]=T(\k v_i)$. 
We may suppose that $\chi_i(f_i)=1$ or $\mathrm{ord}(\chi_i(f_i))=1$,
so long as we then let $e_i$ positive.
\item[(2)]
Let $f_i=g_i^{-1}$,\ $1\le i \le n$. 
Assume $\mathrm{char}\, \k = 0$, and suppose that $L$ and $R$ are the Nichols algebras of $V$ and of $W$,
respectively. 

As was proved by Andruskiewitsch and Schneider
(see \cite[Theorem 4.6 (ii)]{AS1})
and was later strengthened by Heckenberger \cite{H},
$L$ and $R$ are finite-dimensional, 
if the $n \times n$ matrix $\big(q_{ij}\big)$ given by $q_{ij}=\chi_j(g_i)$
satisfies
\[ q_{ij}q_{ji}= q_{ii}^{a_{ij}} \ \, \text{for all} \ \, 1\le i,j \le n \]
for some $n \times n$ Cartan matrix $\big(a_{ij}\big)$, and if for all $1\le i \le n$, 
the order $\mathrm{ord}(q_{ii})$
of $q_{ii}$ is larger than $\mathrm{max}\{1,-a_{ij}\mid 1 \le j \le n \}$. 

Let $\ell > 1$ be an odd integer, and $q \in \k^{\times}$ a primitive $\ell$-th root of $1$. 
Suppose $G=(\mathbb{Z}/(2\ell))^n$, and that $g_i$, $1\le i \le n$,
are the standard generators. Given an $n \times n$ Cartan matrix $\big(a_{ij}\big)$
with its standard symmetrization $\big(d_ia_{ij}\big)$, 
define a symmetric bi-multiplicative map $\tau : G \times G \to \k^{\times}$
by $\tau(g_i,g_j)= q^{d_ia_{ij}}$, $1\le i,j \le n$. 
If the Cartan matrix contains $-3$ as an entry, we add the assumption that
$\ell$ is not divided by $3$. 
Supposing $f_i=g_i^{-1}$, and
that $\chi_i$ are given by $\chi_i(g)=\tau(g_i,g)\, (=\tau(g,g_i))$, we have
$V\in {}_G^G\YD$ and $W \in \YD_G^G$ as above, whose Nichols algebras $L$ and $R$ are
finite-dimensional since the matrix $\big( q^{d_ia_{ij}}\big)$ satisfies the assumption
posed to $\big( q_{ij}\big)$ above. Note that \eqref{eq:chi} holds by the symmetry of $\tau$. 
If we choose $\delta$ again so that $c_i\ne 0$ for all $1\le i \le n$, then the resulting Hopf
algebra $H$ of ours coincides with the quantum Frobenius Kernel constructed by Lusztig \cite{L,Lu}. 
\end{itemize}
\end{example}

\begin{prop}\label{prop:permutableH}
In $H$ the subalgebras in each of (i)--(iv) below are permutable.
\begin{itemize}
\item[(i)]~~$L$\ \, and\ \, $\k G$;\hspace{20mm} $\mathrm{(ii)}$\ \, $\k G$\ \, and\ \, $R$;
\item[(iii)]~~$L$\ \, and\ \, $\k G \ot R$;\hspace{13.5mm} $\mathrm{(iv)}$\ \, $L \ot \k G$\ \, and\ \, $R$.
\end{itemize}
\end{prop}
\begin{proof}
(i), (ii)\ These follow since for every $g \in G$, the inner automorphism $x \mapsto g x g^{-1}$ 
on $H$ stabilizes $L$ and $R$ so that
\[ g x g^{-1}=g \tr x, \ \, x \in L; \quad g x g^{-1}=x \tl g^{-1}, \ \, x \in R.\]

(iv)\ Apply the (bijective) antipode $S$ of $H$ to the both sides of 
$H=(L\lboson \k G) \ot R$, and then use (i). Then we have 
\[
H= S(R) \ot (L\lboson \k G) = S(R) \ot \k G \ot L.
\]
This includes $S(R)\ot \k G \, (=\k G \rboson R)$ as a Hopf subalgebra, which is tensorially 
presented as $R \otimes \k G$ by (ii). The desired result now follows. 

(iii)\ One can argue just as above. 
\end{proof}

\begin{rem}\label{rem:infinite_case}
Suppose that the base field $\k$ is arbitrary, and remove from $G$ (resp., $V$, $W$, $L$ and $R$) the
assumption that it is finite (resp., they are finite-dimensional). 
Then Propositions \ref{prop:def_of_H} and \ref{prop:permutableH} remain to hold since
the proofs above are valid. By saying 
in (i), for example, of Proposition \ref{prop:permutableH} that
$L$ and $\k G$ permutable, we may mean that $\k G \ot L$ is mapped isomorphically onto $L(\k G)$
by the product map. 
\end{rem}

Let us return to the situation posed at the beginning of this subsection. 
Recall that $L$ and $R$ are $\mathbb{N}$-graded Hopf-algebra objects, which
have the counits as augmentations.  
The augmentation ideals $L^+=\bigoplus_{n>0}L(n)$ and 
$R^+=\bigoplus_{n>0}R(n)$ are seen to be $G$-stable, being obviously nilpotent. 
Let
\[
A := L \lboson \k G,\quad B:= \k G \rboson R 
\]
in $H$; see \eqref{eq:tensor_pres_H}.
Then Conditions (I) and (II) in Section \ref{subsec:pattern} are satisfied. Therefore,
\[
I_A:= L^+\ot \k G,\quad I_B:=\k G \ot R^+
\]
are the Jacobson radicals of $A$ and of $B$, respectively, so that
\[
A/I_A = \k G = B/I_B. 
\]
In the present situation Theorem \ref{thm:pattern} reads:

\begin{theorem}\label{thm:Hsimple}
For every $\lambda \in \G$, 
let $\k_B(\lambda)$ denote the one-dimensional left $B$-module given by the algebra map 
\[
B \overset{\mathrm{proj}}{\longrightarrow}  B/I_B =\k G \overset{\k \lambda}{\longrightarrow} \k.
\]  
\begin{itemize}
\item[(1)] The left $H$-module
\[
\mathscr{M}(\lambda)=H \ot_B \k_B(\lambda)
\]
has a unique simple quotient $H$-module, which we denote by $\mathscr{L}(\lambda)$.
\item[(2)]
$\lambda \mapsto \sL(\lambda)$ gives a bijection $\G \isomto H\text{-}\mathsf{Simple}$.   
\end{itemize}
\end{theorem}

%%%%%%%%%%%%%%%%%%%%%%%%%%%%%%%Section 4%%%%%%%%%%%%%%%%%%%%%%%%%%%%%%%%%%%%%
\section{Simple modules over Drinfel'd doubles}\label{sec:double}

Retain $H$ to be the Hopf algebra given by Proposition \ref{prop:def_of_H}.

\subsection{}\label{subsec:Hdual}
Recall from \eqref{eq:tensor_pres_H} the tensorial presentation $H=L \ot \k G \ot R$ of $H$. 

\begin{lemma}\label{lem:H_struc}
We have the following.
\begin{itemize}
\item[(1)] $\ep \ot \mathrm{id} \ot \ep : H=L \ot \k G \ot R\to \k G$ is a coalgebra
map, by which $\k G$ turns into a quotient coalgebra of $H$. 
\item[(2)] Regard $L$ and $R$ as coalgebras with respect to the original structures as
coalgebra objects in ${}_G^G\YD$ and in $\YD_G^G$, respectively. Then
\[
L \overset{\mathrm{id} \ot \ep \ot \ep}{\longleftarrow\hspace{-1.6mm}-} H=L \ot \k G \ot R
\overset{\ep \ot \ep \ot \mathrm{id}}{-\hspace{-1.6mm}\longrightarrow} R 
\]
are coalgebra maps, by which $L$ and $R$ turn into quotient left and respectively, right
$H$-module coalgebras of $H$; see Remark \ref{rem:coideal_subalgebra}. 
\item[(3)] The iterated coproduct $(\mathrm{id}\ot \Delta)\circ \Delta : H \to H\ot H \ot H$, composed
with the tensor product of the three quotient maps $H \ot H \ot H \to L \ot \k G \ot R$,
gives an inverse of the isomorphic product map $L \ot \k G \ot R \isomto H$. 
\end{itemize}
\end{lemma}
\begin{proof}
Parts 1 and 2 are easy to see. Part 3 follows since one sees 
by a direct computation
that if $x \in L$, $g \in G$
and $y \in R$, the product $xgy$ is sent to $x\ot g \ot y$ by the prescribed composite. 
\end{proof}

Let $H^*$ denote the dual Hopf algebra of $H$. 
To present the coproduct of $H$ or of $H^*$,
we will use the following variant of the Heyneman-Sweedler notation \cite[Notation 1.4.2, p.6]{Mon}
\[
\Delta(h) = h_{(1)} \ot h_{(2)},\quad (\mathrm{id}\ot \Delta)\circ \Delta(h) = h_{(1)}\ot h_{(2)} \ot h_{(3)}. 
\]
We let 
$\l \ , \ \r : H^* \times H \to \k, \ \l \a, h \r =\a(h)$ denote the canonical pairing. 
As a vector space, $H^*$ is canonically identified so as
\begin{equation}\label{eq:Hdual_present}
H^*=L^* \ot \k \G \ot R^*.
\end{equation}
 
By duality (see Remark \ref{rem:coideal_subalgebra}), 
Lemma \ref{lem:H_struc} yields the following:

\begin{corollary}\label{cor:Hdual_struc}
In $H^*$, $\k \G$ is a subalgebra, $\L$ is a left coideal subalgebra, and 
$\R$ is a right coideal subalgebra. Moreover, Eq.~\eqref{eq:Hdual_present} gives 
a tensorial presentation of the algebra $H^*$. 
\end{corollary}

\begin{lemma}\label{lem:dual_Hopf_object}
We have the following.
\begin{itemize}
\item[(1)]
$\L=\bigoplus_{n\ge 0}L(n)^*$ and $\R=\bigoplus_{n\ge 0}R(n)^*$ are naturally 
$\mathbb{N}$-graded Hopf-algebra objects in ${}^{\G}_{\G}\YD$ and in $\YD_{\G}^{\G}$, respectively,
such that $\L(0)=\k =\R(0)$. 
\item[(2)]
If $\mathrm{char}\, \k = 0$, then $\L$ and $\R$ are both Nichols algebras.
\item[(3)]
If $\mathrm{char}\, \k > 0$, then $\L$ is not necessarily generated by 
$\L(1)=V^*$; $\L$ is generated 
by $\L(1)$, if and only if $\L$ is Nichols, if and only if $L$ is Nichols. 
The same holds for $\R$.
\end{itemize}
\end{lemma}
\begin{proof}
(1)\ 
To see this, identify the $G$-gradings on $L$ and on $R$ 
with left and, respectively, right $\k G$-comodule structures
\[ L \to \k G \ot L, \quad R \to R \ot \k G, \]
and dualize all the structure maps; the $\k G$-comodule structure, for example, on $L$ sends every homogeneous
element $x \in L_g$, $g \in G$, to $g \ot x$. 

(2)\ This follows from Theorem \ref{thm:in_char_zero}, by duality.

(3)\ The three conditions all are equivalent to that $L$ is strictly graded \cite[Sect. 11.2]{Sw} as an $\mathbb{N}$-graded
coalgebra, or in other words, $L(1)$ coincides with the space of all primitives in $L$. 
\end{proof}

From the Hopf-algebra objects above there arise the algebras
\begin{equation}\label{eq:smash_prod}
\L \rtimes \k \G,\quad \k \G \ltimes \R
\end{equation}
of smash product, in particular.

\begin{prop}\label{prop:Hdual_suppl}
For \ $\L \ot \k \G$, \ $\k \G \ot \R$ \ 
and \ $\k \G$ \ 
in \ 
$H^*=L^* \ot \k \G \ot R^*$, we have the following.
\begin{itemize}
\item[(1)] $\L \ot \k \G$ is a left coideal subalgebra of $H^*$, which is $\L \rtimes \k \G$
as an algebra. 
\item[(2)] $\k \G \ot \R$ is a right coideal subalgebra of $H^*$, which is $\k \G \ltimes \R$
as an algebra. 
\item[(3)] We have
\begin{equation}\label{eq:coprod_Gdual}
\Delta(\k \G)\subset (\k \G \ot \R)\ot (\L \ot \k \G). 
\end{equation}
\end{itemize}
\end{prop}
\begin{proof}
For Part 1, dualize the fact that $L\otimes \k G=H/HR^+$ is a quotient left $H$-module coalgebra
of $H$, which is the smash coproduct $L \lcosmash \k G$ as a coalgebra. 
For Part 2, argue similarly. 
Part 3 follows from 1 and 2. 
\end{proof}

\begin{corollary}\label{cor:permutable1}
$L^*$ and $\k \G$, as well as $\k \G$ and $\R$, are permutable in $H^*$.
\end{corollary}

\subsection{}\label{subsec:double}
Let $(H^*)^{\mathrm{cop}}$ denote the Hopf algebra obtained from $H^*$ by opposing the coproduct.
The Drinfel'd double $D(H)$ of $H$ is the Hopf algebra determined by the
properties: (i)~~It includes $(H^*)^{\mathrm{cop}}$ and $H$ as Hopf algebras, and is tensorially presented as
\[
D(H) = (H^*)^{\mathrm{cop}}\ot H.
\]
(ii)~~The product is determined by 
\begin{equation}\label{eq:formula1}
h\a =
\l \as, \hi \r \, \an \hn \, \l \ai, \Sm (\hs) \r,\quad h \in H,\ \a \in H^*
\end{equation}
or equivalently, by
\begin{equation}\label{eq:formula2}
\a h =
\l \as, 
\Sm(\hi) \r \, \hn \an \, \l \ai, \hs \r,\quad \a \in H^*,\ h \in H. 
\end{equation}
See \cite[Sect.~10.3]{Mon} or \cite[Chap.~IX, Sect.~4]{Ka}. 
Here $S^{-1}$ denotes the inverse of the antipode $S$ of $H$, which is indeed bijective. 
We emphasize that $\Delta(\a)=\ai \ot \an$ represents the coproduct of $H^*$, not of $(H^*)^{\mathrm{cop}}$.

By using the antipode one sees that the subalgebras $H^*$ and $H$ are permutable in $D(H)$. 
Obviously, we have the tensorial presentation of $D(H)$:
\begin{equation}\label{eq:D(H)_present}
D(H) = \L \ot \k \G \ot \R \ot L \ot \k G \ot R. 
\end{equation}

\begin{prop}\label{prop:inner_auto}
Given $g \in G$, the inner automorphism $x \mapsto gxg^{-1}$ on $D(H)$ stabilizes
each tensor factor on the right hand-side of \eqref{eq:D(H)_present}. The restricted
actions on $\k \G$ and $\k G$ are trivial, and those on $\L$ and $\R$ are given by
\begin{equation}\label{eq:G-actions}
\a \mapsto \l \a_{(1)}, g^{-1}\r\, \a_{(2)}\quad \text{and}\quad
\b \mapsto \b_{(1)}\, \l \b_{(2)}, g \r,
\end{equation}
respectively, where $\a \in \L$ and $\b \in \R$.  
\end{prop}
\begin{proof}
For $\k G$ this is obvious. For $L$ and $R$ this was shown in the proof of Proposition \ref{prop:permutableH}. 
One sees from \eqref{eq:formula1} that the inner automorphism 
stabilizes $H^*$. Moreover, the restricted action on $H^*$ is dualized to the inner automorphism 
$h \mapsto g^{-1}hg$ on $H$, which induces automorphisms on the quotient coalgebras
$L$, $\k G$ and $R$; see Lemma \ref{lem:H_struc}. The induced action on $\k G$ is trivial, while
that on $L$ (resp., on $R$) is the action by $g^{-1}\in H$ on the quotient left $H$-module $L$ of $H$ 
(resp., the action by $g\in H$ on the quotient right $H$-module $R$ of $H$). 
This proves the remaining. 
\end{proof}

\begin{corollary}\label{cor:permutable2}
We have the following.
\begin{itemize}
\item[(1)] $\k G$ and any of the tensor factors above are permutable. 
\item[(2)] $\k \G \ot \k G$ in $D(H)$ is a subalgebra, and is  
the group algebra $\k(\G \times G)$. 
\end{itemize}
\end{corollary}

\begin{prop}\label{prop:permutableD}
In $D(H)$ the subalgebras in each of (i)--(vi) below are permutable.
\begin{itemize}
\item[(i)]~~$H^*$\ \, and\ \, $\k G \ot R$;\qquad $\mathrm{(ii)}$\ \, $H^*$\ \, and\ \, $L \ot \k G$;
\item[(iii)]~~$\L$\ \, and\ \, $R$;\hspace{16mm} $\mathrm{(iv)}$\ \, $\R$\ \, and\ \, $L$;
\item[(v)]~~$\L\ot \k \G$\ \, and\ \, $R$;\hspace{7mm} $\mathrm{(vi)}$\ \, $\k \G \ot \R$\ \, and\ \, $L$.
\end{itemize}
\end{prop}
\begin{proof}
(i), (ii)\ One sees from the formulas \eqref{eq:formula1}, \eqref{eq:formula2} that
the Hopf subalgebras $H^*$ and $\k G \ot R$, as well as $H^*$ and $L \ot \k G$, 
are permutable.

(iv)\ Let $h \in L$ and $\a \in \R$. Recall that $L$ is a left coideal of $L\ot \k G\, (\subset H)$, and $\R$ 
is a right coideal of $H^*$. Since one then sees
\begin{equation}\label{eq:aSm}
\l \a, \Sm(h) \r = \ep(\a)\, \ep(h) = \l \a, h \r, 
\end{equation}
there result from \eqref{eq:formula1}, \eqref{eq:formula2} the formulas
\begin{equation}\label{eq:ha}
h \a = \l \an , \hi \r \, \ai \hn,
\end{equation}
\begin{equation}\label{eq:ah}
\a h = \l \an , \Sm(\hi) \r \, \hn \ai,
\end{equation}
which prove $h \a \in \R L$ and $\a h \in L\R$, respectively.

(vi)\ Let $h \in L$ and $\a \in \k \G \ot \R$. We claim that the last two formulas remain true; 
they prove the desired result. To prove the claim, it is enough to see that the the equalities in 
\eqref{eq:aSm} hold. The second equality obviously holds. We wish to prove the first one. 
We may suppose that the element $h$ of the (Hopf-algebra) object $L$ in ${}^G_G\YD$ is $G$-homogeneous,
being of degree $g\, (\in G)$, say. Since the counit on $L$ is $G$-graded, it follows that $\ep(h)=0$ unless $g=1$,
whence $\ep(h)g^{-1}=\ep(h)1$. 
Let $s : L \to L$ denote the (bijective) antipode of the Hopf-algebra object, and $s^{-1}$
its inverse. Since we see from $S(h)=g^{-1}s(h)$ that
\[
S^{-1}(h) = s^{-1}(h) \ot g^{-1} \in L\ot \k G,
\]
it follows that
$(\ep \ot \mathrm{id})(S^{-1}(h))=\ep(h)g^{-1}=\ep(h)1$, 
proving the desired equality. 

(iii), (v)\ \, One can argue just as for (iv), (vi). 
\end{proof}

\begin{lemma}\label{lem:tensor_pres_D}
We have the tensorial presentation
\[
D(H)= L^*\ot R \ot \k(\G \times G) \ot L \ot R^*
\]
of $D(H)$, in which $\L \ot R$ and $L \ot \R$ are both subalgebras
permutable with $\k (\hat{G} \times G)$.  
\end{lemma}
\begin{proof}
Recall Corollary \ref{cor:permutable2}. The one sees from
Propositions \ref{prop:permutableH} and \ref{prop:permutableD} (i), (iii), (iv) that
\begin{align*}
D(H)&=H^*\ot R\ot \k G \ot L\\
&= R \ot \k G \ot H^* \ot L\\
&= L^* \ot R \ot \k G \ot \k \hat{G} \ot L \ot R^*,
\end{align*}
which gives the desired tensorial presentation. 
Again by Proposition \ref{prop:permutableD} (iii), (iv), 
$\L \ot R$ and $L \ot \R$ are both subalgebras.
By Proposition \ref{prop:permutableD} (v) and Corollary \ref{cor:permutable1} we have
\begin{align*} 
\L  \ot R \ot \k \hat{G} \ot \k G &= R \ot (\L \ot \k \hat{G}) \ot \k G\\
&= (\L \ot \k \hat{G})\ot R \ot \k G\\
&= \k \hat{G} \ot \L \ot R \ot \k G\\
&= \k \hat{G} \ot \k G \ot \L \ot R. 
\end{align*}
This shows that $\L  \ot R$ is permutable with $\k (\hat{G}\times G)$. 
One can argue similarly to prove that $L  \ot \R$ is as well.  
\end{proof}

We wish to see that $D(H)$, tensorially presented as above, is in the same situation as
$\Lambda=L\ot \k G \ot R$ in Section \ref{subsec:pattern}; we let $\hat{G}\times G$
play the role of the group $G$ in $\Lambda$. 

Lemma \ref{lem:tensor_pres_D} ensures that in $D(H)$,
\begin{equation}\label{eq:AB}
A:= \L \ot R \ot \k (\G \times G),\quad B:= \k(\G \times G)\ot L\ot \R. 
\end{equation}
define subalgebras which include the subalgebras $\L \ot R$ and $L \ot \R$, respectively,
permutable with the group algebra $\k (\G \times G)$.

Recall from Lemma \ref{lem:dual_Hopf_object} (1) that $\L$ and $\R$ are 
$\mathbb{N}$-graded Hopf-algebra objects, which have the counits as augmentations. 
Using the augmentation ideals $(\L)^+=\bigoplus_{n>0}L(n)^*$ and $(\R)^+=\bigoplus_{n>0}R(n)^*$, 
we define subspaces $I_A \subset A$ and $I_B \subset B$ by
\begin{align*}
I_A &:= (\L)^+ \ot R \ot \k(\G \times G) + \L \ot R^+ \ot \k(\G \times G),\\
I_B &:= \k(\G \times G) \ot L \ot (\R)^+ + \k(\G \times G) \ot L^+ \ot \R.
\end{align*}

\begin{prop}\label{prop:ideal}
$I_A$ and  $I_B$ are nilpotent ideals of $A$ and of $B$, respectively, such that
\[ A/I_A \simeq \k(\G \times G)\simeq B/I_B. \]
\end{prop}
\begin{proof}
We concentrate on $I_B$, since we can do with $I_A$ similarly. It suffices to prove that
$I_B$ is an nilpotent ideal, since it immediately implies $B/I_B \simeq \k(\G \times G)$. 

Let
\[
J_1:= \k(\G \times G) \ot L \ot (\R)^+,\quad J_2:=\k(\G \times G) \ot L^+ \ot \R. 
\]
The proof is divided into two steps. 
The first step is to prove that $J_1$ is an ideal of $B$. 
We claim that the inner automorphisms on $\R$ 
by $G$ and by $\G$ are both compatible with the counit, whence they stabilize $(\R)^+$.
For the inner automorphisms by $G$ this is seen from \eqref{eq:G-actions}.
For those by $\hat{G}$ this follows since the stability is obvious in the bosonization
$\k \hat{G} \rboson \R$, and this result proves the claim by Proposition \ref{prop:Hdual_suppl} (2). 
By the claim just proven
it remains to prove $(\R)^+L\subset L(\R)^+$. This follows if one applies $\mathrm{id} \ot \ep$ to
the element $\l \an, \Sm(\hi) \r \, \hn  \ai \in L \ot \R$ on the right hand-side of \eqref{eq:ah},
and then obtains $\ep(\a)h$, which is zero if $\a \in (\R)^+$.
We remark: (1)~$J_1$, which has just been proved to be an ideal, is nilpotent since $(\R)^+$ is.

The quotient algebra $B/J_1$ has the tensorial presentation $B/J_1= \k(\G \times G) \ot L$. 
The second step of the proof is
to prove the natural image $\k(\G \times G) \ot L^+$ of $J_2$ (or of $I_B$) is an ideal of $B/J_1$.
Since $L^+$ is stable under the inner automorphisms by $G$, it remains to prove that
$L^+(\k \G) \subset (\k \G) L^+$ in $B/J_1$. Let $h \in L$ and $\a \in \G$.
We claim
\begin{equation*}\label{eq:at_revision}
h \a = \a h_{(2)}\, \l \a, h_{(1)} \r \ \, \text{in}\ \, B/J_1. 
\end{equation*}
If $h \in L^+$, the right-hand side is killed by 
$\mathrm{id}\otimes \ep : \k(G\times \G) \ot L \to \k(G \times \G)$,
and so the desired $h \a \in (\k \G)L^+$ will follow. 
Let $\pi = \mathrm{id} \ot \ep : \k \G \otimes \R \to \k \G$ denotes the natural projection.
Since Eq.~\eqref{eq:ha}
holds (see the proof of Part vi of Proposition \ref{prop:permutableD}), it follows that
the product
$h \a$ in $B/J_1$ coincides with $\pi(\ai)\hn\, \l \an,\hi \r \in \k \G \otimes L$.
We see from \eqref{eq:coprod_Gdual} that $\pi(\ai) \ot \an$
coincides with the coproduct $\Delta(\pi(\a))$ of $\pi(\a)$ on the quotient coalgebra
$H^*/H^*(\R)^+=L^* \ot \k \G$. Since this last is indeed a Hopf algebra which includes $\k \G$
as a Hopf subalgebra, the coproduct $\Delta(\pi(\a))$ equals $\a \ot \a$; this proves the claim. 
We remark: (2)~$I_B/J_1\, (= \k(\G \times G)\ot L^+)$, which has just been proved to be an ideal of $B/J_1$,
is nilpotent since $L^+$ is. 

We conclude that $I_B$ is an ideal of $B$; it is nilpotent by the remarks (1), (2).
\end{proof}

See Remark \ref{rem:added_in_revision} below for a simplified argument proving the proposition.

\begin{rem}
To confirm that $D(H)$ is in the same situation as the $\Lambda$ in Section \ref{subsec:pattern},
we see from Proposition \ref{prop:ideal} that 
the subalgebra $\L \ot R$ of $A=\L \ot R \ot \k(\G \times G)$
is augmented with respect to the augmentation ideal
\[
(\L \ot R)^+ := (\L)^+ \ot R + \L \ot R^+\, (=I_A \cap (\L \ot R)), 
\]
which is nilpotent, and satisfies 
$\k (\hat{G}\times G)\, (\L \ot R)^+=(\L \ot R)^+\, \k (\hat{G}\times G) \, (=I_A)$. 
Similarly, the subalgebra $L \ot \R$ of 
$B=k(\G \times G)\ot L \ot \R$ is augmented with respect to the augmentation ideal
\[
(L \ot \R)^+ := L^+ \ot \R + L \ot (\R)^+\, (=I_B \cap (L \ot \R)), 
\]
which is nilpotent, and satisfies 
$(L \ot \R)^+\, \k (\hat{G}\times G)=\k (\hat{G}\times G)\, (L \ot \R)^+\, (=I_B)$.
\end{rem}

Recall that in the present situation, the group playing the role of the $G$ 
in Section \ref{subsec:pattern} is $\hat{G}\times G$, whose dual group
$\widehat{\G \times G}$ is $G \times \G$. 
Theorem \ref{thm:pattern} now reads:

\begin{theorem}\label{thm:D(H)simple}
For every $\lambda \in G \times \G$, 
let $\k_B(\lambda)$ denote the one-dimensional left $B$-module given by the algebra map 
\[
B \overset{\mathrm{proj}}{\longrightarrow} B/I_B =\k(\G \times G)\overset{\k \lambda}{\longrightarrow} \k.
\] 
\begin{itemize}
\item[(1)] The left $D(H)$-module
\[
\mathscr{M}(\lambda)=D(H)\ot_B \k_B(\lambda)
\]
has a unique simple quotient $D(H)$-module, which we denote by $\mathscr{L}(\lambda)$.
\item[(2)]
$\lambda \mapsto \sL(\lambda)$ gives a bijection $G \times \G \isomto D(H)\text{-}\mathsf{Simple}$.   
\end{itemize}
\end{theorem}

As was pointed out by the referee, Part 2 can be formulated as a bijection 
$D(\k G)\text{-}\mathsf{Simple} \isomto D(H)\text{-}\mathsf{Simple}$. This is the same
phenomenon as the one that was discovered by \cite{PV} and others, when $H$ is the bosonization of
a finite-dimensional Nichols algebra over an abelian or non-abelian finite group.

\begin{rem}[added in revision]\label{rem:added_in_revision}
After submitting an earlier version of the present paper the authors found the article \cite{BT} by
Bellamy and Thiel;
they develops highest weight theory of $\mathbb{Z}$-graded modules over a finite-dimensional $\mathbb{Z}$-graded
algebra which has what they call \emph{triangular decomposition} \cite[Definition 3.1]{BT}. 
This last notion  
was essentially introduced by Holmes and Nakano \cite{HN}, and generalizes 
ours (see Section \ref{subsec:pattern})
in the sense that it is specialized so as to well apply to our objective, as follows. 

Let $\Lambda$ be a finite-dimensional $\mathbb{Z}$-graded algebra including a group algebra $\k G$
as its neutral component $\Lambda(0)$. Suppose in addition: (1)~$\Lambda$ includes a non-positively
graded subalgebra $L=\bigoplus_{n\le 0}L(n)$ and a non-negatively graded subalgebra $R=\bigoplus_{n\ge 0}R(n)$,
each of which is permutable (see Section \ref{subsec:tensor_product}) with $\k G$, and 
(2)~we have $\Lambda=L\ot \k G \ot R$ as in \eqref{eq:TRD}. 
We then necessarily have: (i)~$L(0)=R(0)=\k$, and
(ii)~$L^+=\bigoplus_{n<0}L(n)$ and $R^+=\bigoplus_{n>0}R(n)$ are nilpotent (augmentation) ideals.
Therefore, Conditions (I) and (II) in Section \ref{subsec:pattern}
are necessarily satisfied. As for (II), the first equality $(\k G)L^+=L^+(\k G)$, for example,
holds since the both sides are the sum of all positive components in $\k G \ot L=L \ot \k G$.

A close look at our arguments above, combined with some additional remarks below, will show: 
\emph{The finite-dimensional Hopf algebra $H$ given
by Proposition \ref{prop:def_of_H} and its Drinfel'd double $D(H)$ discussed in the present section 
fit in with the situation above.} For this we regard $L$ as non-positively graded,
changing the signs of degrees.
In addition we regard $L^*$ and $R^*$ as non-negatively and non-positively graded,
respectively, so that the relevant pairings preserves the grading, with $\k$ trivially graded, $\k=\k(0)$. 

Indeed, as is seen from \eqref{eq:wv}, $H$ is a $\mathbb{Z}$-graded Hopf algebra, and so $H^*$ is as well. 
Now, the desired result for $H$ is easy to see. One sees from \eqref{eq:formula1} that
$D(H)$ is a $\mathbb{Z}$-graded Hopf algebra. 
To see the desired result for $D(H)$, it suffices to add to the paragraph containing
\eqref{eq:AB} the following immediate consequence of Proposition \ref{prop:permutableD} (iii), (iv):
$L^*R\, (=L^*\ot R)$ (resp., $LR^*\, (=L\ot R^*)$) is a non-negatively (resp., non-positively)
graded subalgebra of $D(H)$. (To make the situation completely fit in with the above one we have to 
change the signs of degrees. 
But this is inessential.)
With this having seen, Proposition \ref{prop:ideal} turns to be obvious. 

Bellamy and Thiel \cite{BT} list seven classes of finite-dimensional $\mathbb{Z}$-graded algebras
with triangular decomposition, to which their own highest weight theory can apply. It includes 
the class of Lusztig's quantum Frobenius Kernels. Our contribution is to have 
widened the last class to those $H$ as above, and to have added $D(H)$ to their list. 
\end{rem}

\section*{Acknowledgments}
The work was supported by JSPS Grant-in-Aid for Scientific Research (C), 26400035 and 17K05189. 
The authors thank the referee for his/her helpful comments.

\end{document}